\documentclass[onefignum,onetabnum]{siamart171218}

\usepackage{amsmath,amsfonts,amssymb,bm,algorithmic,autonum}
\usepackage[normalem]{ulem}

\allowdisplaybreaks

\newcommand{\sxy}[2]{\langle#1,#2\rangle}
\newcommand{\ssxy}[2]{\langle\langle#1,#2\rangle\rangle}
\newcommand{\no}[1]{|||#1|||}
\newcommand{\bcal}[1]{\bm{\mathcal{#1}}}
\newcommand{\Ce}{\mathbb{C}}
\newcommand{\Er}{\mathbb{R}}
\newcommand{\Es}{\mathbb{S}}

\newsiamremark{remark}{Remark}
\newsiamremark{notation}{Notation}
\headers{Prescribing convergence behavior of block Arnoldi and GMRES}{M. Kub\'{i}nov\'{a}, K. Soodhalter}

\title{Admissible and attainable convergence behavior of~block Arnoldi and GMRES
	\thanks{This version dated \today.
		\funding{The work of M.\,K. was supported by the Czech Academy of Sciences through the project L100861901 (Programme for promising human resources -- postdocs) and by the Ministry of Education, Youth and Sports of the Czech Republic through the project LQ1602 (IT4Innovations excellence in science).}}}

\author{Marie Kub\'{i}nov\'{a}\thanks{The Czech Academy of Sciences, Institute of Geonics, Ostrava, Czech Republic
		(\email{marie.kubinova@ugn.cas.cz}, \href{http://www.ugn.cas.cz/\%7Ekubinova/}{http://www.ugn.cas.cz/\textasciitilde kubinova/}).}
	\and Kirk M. Soodhalter\thanks{Trinity College Dublin, The University of Dublin, Dublin, Ireland
		(\email{ksoodha@maths.tcd.ie}, \href{https://math.soodhalter.com/}{https://math.soodhalter.com/}).}}

\usepackage{amsopn}

\usepackage{todonotes}

\ifpdf
\hypersetup{
  pdftitle={Prescribing convergence behavior of block Arnoldi and GMRES},
  pdfauthor={M. Kubinova, K. Soodhalter}
}
\fi

\begin{document}

\maketitle

\begin{abstract}
  It is well-established that any non-increasing convergence curve is possible for  GMRES and a family of pairs $(A,b)$ can be constructed for which GMRES exhibits a given convergence curve with $A$ having arbitrary spectrum. No analog of this result has been established for block GMRES, wherein multiple right-hand sides are considered. By reframing the problem as a single linear system over a ring of square matrices, we develop convergence results for block Arnoldi and block GMRES. In particular, we show what convergence behavior is admissible for block GMRES and how the matrices and right-hand sides producing any admissible behavior can be constructed. Moreover, we show that the convergence of the block Arnoldi method for eigenvalue approximation can be almost fully independent of the convergence of block GMRES for the same coefficient matrix and the same starting vectors.
\end{abstract}



\textit{The second author would like to dedicate this work to the late Richard Timoney, who was very generous with his time in explaining the basics of $^\ast$-algebras to an uninitiated new colleague.}

\section{Introduction}
The celebrated GMRES algorithm \cite{Saad1986GMRES} is an effective, widely-used iterative method for solving linear systems 
\begin{equation}
A\bm{x} = \bm{b}, \quad A\in\Ce^{m\times m}, \quad \bm{b}\in\Ce^{m},
\end{equation}
with non-Hermitian coefficient matrices. It has been shown that, in contrast to the case of Hermitian matrices, 
for non-Hermitian coefficient matrices, we cannot guarantee certain convergence speed based solely on our knowledge of the spectrum. In particular, for any given non-increasing sequence of positive numbers $f_{0}\geq f_{1}\geq\cdots\geq f_{n-1}$, one can construct a non-Hermitian matrix with
arbitrary spectrum for which GMRES produces residuals whose norms correspond
to this sequence \cite{Greenbaum1996Any}.  
In \cite{Meurant2012GMRES}, it was observed this does not mean eigenvalues are meaningless for the convergence of GMRES applied to non-normal problems.  This result simply
establishes one extreme of what role eigenvalues can play in the \textit{residual} convergence of GMRES.  

If instead we solve multiple systems with the same $A$ (i.e., a system with multiple right-hand sides) 
\begin{equation}
AX = B, \quad A\in\Ce^{m\times m}, \quad B\in\Ce^{m\times s},
\end{equation}
a direct generalization of GMRES called block GMRES exists which produces approximate solutions simultaneously for all right-hand sides.  
We would like to explore if block GMRES residual convergence admits a similar characterization as in \cite{Greenbaum1996Any}.

Most often, analysis of block GMRES takes the view of the method as a minimization of each individual residual over a sum of spaces.  This enables some basic convergence analysis but such analysis fails to capture the full picture of block GMRES behavior, which is influenced by the interaction between the different right-hand sides. The nature of this interaction is quite difficult to describe when considering block GMRES as a method treating a collection of scalar linear systems.

Indeed, some authors have taken a different approach, discussing these methods in terms of vector blocks in $\Ce^{m\times s}$, namely \cite{Frommer2017Block} and \cite{Simoncini1996Convergence}.
We demonstrate here that embracing a totally block view of this iteration greatly simplifies analysis of block GMRES and allows us to obtain clean convergence results.

The aim of the paper is to extend well-known GMRES/Arnoldi convergence results
developed in \cite{Arioli1998Krylov,DuintjerTebbens2012Any,Greenbaum1994Matrices,Greenbaum1996Any} (along with many excellent follow-up papers by subsets of the same authors) to the block case using the framework of \cite{Frommer2017Block,Frommer2019Block}. 
\medskip

The paper is organized as follows. \Cref{sec:preliminaries} introduces the block methods that will be analyzed and the framework used to perform the analysis. In \Cref{sec:admissible_conv}, we generalize the notion of non-increasing convergence curve to the block setting.  \Cref{sec:prescribing_conv} provides a characterization of matrices and starting vectors exhibiting prescribed convergence behavior, including discussion of the spectral properties of the obtained coefficient matrices. We summarize our results and formulate open questions in \Cref{sec:conclusion}. Throughout the paper, we assume exact arithmetic.
We also introduce some specialized notation for this
setting in \Cref{notation.caligraphic}.

\section{Preliminaries}\label{sec:preliminaries}

In each step, block Krylov subspace methods look for an approximation of each individual solution in the space
\begin{equation}\label{eq:krylov_subspace}
\mathcal{K}_k(A,R_0) \equiv \text{colspan}\{R_0,AR_0,\ldots,A^{k-1}R_0\} \mbox{\ \ with\ \ } R_0=B-AX_0.
\end{equation}
For simplicity of presentation, we assume $X_0 = 0$ leading to $R_0=B$ throughout the paper. 

If the columns of the block vectors $V_1,\ldots, V_k$, $V_j\in\Ce^{m\times s}$, form a basis of the subspace $\mathcal{K}_k(A,B)$, we say that $V_1,\ldots, V_k$ is a \emph{block} basis of $\mathcal{K}_k(A,B)$. The block solution $X_k$ can be represented as a \emph{block} linear combination of these block vectors:
\begin{equation}\label{eq:lin_comb}
X_k = \sum_{i=1}^{k}V_i\,D_i, \quad D_i\in\Ce^{s \times s},
\end{equation}
and we say in this setting that $X_k\in\text{blockspan}\{B,AB,\ldots,A^{k-1}B\}$.
The particular choice of $\{D_i\}_{i=1}^k$ is defined by the conditions which the method imposes on the residual $R_k = B-AX_k$; cf. \Cref{sec:block_Arnoldi}.  For an overview of block Krylov subspace methods, see, e.g., \cite[sec. 6.12]{Saad2003Iterative}.

Comparing \cref{eq:lin_comb} with the standard definition of a linear combination of a set of vectors, we see that the $s\times s$ matrices here play the role of complex numbers $\Ce$. Let $n = \lceil\frac{m}{s}\rceil$. Padding the matrix and the right-hand side by zeros as follows
\begin{equation}\label{eq:padding}
\widehat{A} \longrightarrow \begin{bmatrix}A&0\\0&I\end{bmatrix}\in\Ce^{ns\times ns}, \quad	\widehat{B} \longrightarrow \begin{bmatrix}	B\\0\end{bmatrix}\in\Ce^{ns\times s}, \quad	\widehat{X} \longrightarrow \begin{bmatrix}	X\\0\end{bmatrix}\in\Ce^{ns\times s},
\end{equation}
the unknown $X$ and block GMRES approximations thereof do not change and we can view the new matrix also as an $n \times n$ array of $s\times s$ matrices, and similarly the new right-hand side and the solution as a vector of length $n$ of $s\times s$ matrices. Thus, to simplify the presentation and without loss of (except, perhaps, at the final iteration), 
we assume hereafter that $m = ns$, and decompose $A$, $B$, and $X$ into   blocks, which we demonstrate on the following example.

Let $m=6$ and $s=2$, then $n=3$, and we can write
\begin{equation}
\underbrace{\left[{\footnotesize \begin{array}{c|c|c}
	\begin{matrix}
	a_{1,1}&a_{1,2}\\
	a_{2,1}&a_{2,2}\\
	\end{matrix}
	&
	\begin{matrix}
	a_{1,3}&a_{1,4}\\
	a_{2,3}&a_{2,4}\\
	\end{matrix}
	&
	\begin{matrix}
	a_{1,5}&a_{1,6}\\
	a_{2,5}&a_{2,6}\\
	\end{matrix}
	\\\hline
	\begin{matrix}
	a_{3,1}&a_{3,2}\\
	a_{4,1}&a_{4,2}\\
	\end{matrix}
	&
	\begin{matrix}
	a_{3,3}&a_{3,4}\\
	a_{4,3}&a_{4,4}\\
	\end{matrix}
	&
	\begin{matrix}
	a_{3,5}&a_{3,6}\\
	a_{4,5}&a_{4,6}\\
	\end{matrix}
	\\\hline
	\begin{matrix}
	a_{5,1}&a_{5,2}\\
	a_{6,1}&a_{6,2}\\
	\end{matrix}
	&
	\begin{matrix}
	a_{5,3}&a_{5,4}\\
	a_{6,3}&a_{6,4}\\
	\end{matrix}
	&
	\begin{matrix}
	a_{5,5}&a_{5,6}\\
	a_{6,5}&a_{6,6}\\
	\end{matrix}
	\end{array}}\right]}_{\left[\begin{array}{cccc}
A_{1,1}
&
A_{1,2}
&
A_{1,3}
\\
A_{2,1}
&
A_{2,2}
&
A_{2,3}
\\
A_{3,1}
&
A_{3,2}
&
A_{3,3}
\end{array}
\right]}
\underbrace{\left[{\footnotesize \begin{array}{cc}
	x_{1,1}&x_{1,2}\\
	x_{2,1}&x_{2,2}\\
	\hline
	x_{3,1}&x_{3,2}\\
	x_{4,1}&x_{4,2}\\
	\hline
	x_{5,1}&x_{5,2}\\
	x_{6,1}&x_{6,2}
	\end{array}}\right]}_{\left[\begin{array}{c}
X_1\\
X_2\\
X_3
\end{array}\right]}
= 
\underbrace{\left[{\footnotesize \begin{array}{cc}
	b_{1,1}&b_{1,2}\\
	b_{2,1}&b_{2,2}\\
	\hline
	b_{3,1}&b_{3,2}\\
	b_{4,1}&b_{4,2}\\
	\hline
	b_{5,1}&b_{5,2}\\
	b_{6,1}&b_{6,2}
	\end{array}}\right]}_{\left[\begin{array}{c}
B_1\\
B_2\\
B_3
\end{array}\right]},
\end{equation}
where $A_{i,j}$, $X_{j}$, and $B_{i}$ are all $2\times 2$ matrices.

In this paper, we restrict our analysis to the case when $\dim(\mathcal{K}_n(A,B)) = ns$.

\subsection{Framework}\label{sec:framework}
\sloppy{To generalize the results from standard Krylov subspace methods to block ones, we follow \cite[sec. 2]{Frommer2017Block} and replace $\Ce$ by the non-commutative $^*$-algebra $\Es$ of complex $s\times s$ matrices.}\footnote{The authors in \cite{Frommer2017Block} considered various other subrings of $\Es$ to classify different block-type Krylov subspace methods in a common framework.} We observe that objects and operations over $\Ce$ not relying on commutativity have counterparts in $\Es$ that are relevant for the analysis of block Krylov subspace methods; some of them are shown in \Cref{tab:framework}. We use the notation $\mathbb{S} \simeq \mathbb{C}^{s\times s}$ to say that the elements of $\Es$ are from $\Ce^{s\times s}$ but equipped with additional operations.

We emphasize that to comply with the standard block Krylov subspace method notation, the generalization of positive real numbers denoted by $\Es^+$ corresponds to the upper triangular matrices with positive diagonal entries.\footnote{The standard subalgebra of positive entries is represented by Hermitian positive-definite matrices, whose Cholesky factors have a one-to-one correspondence with elements of $\Es^+$ used here.} This choice becomes particularly important for the $^*$-algebra formulation of the block Arnoldi algorithm, see \Cref{sec:block_Arnoldi}. Note also that in $\Es^n$, similarly to $\Ce^n$, we distinguish between the \emph{block absolute value} $|\cdot|$ and the \emph{block norm} $\no{\cdot}$. Both these operators map to $\Es_0^+$, however the first is defined only for the elements of $\Es$, whereas the second one is defined for an arbitrary vector with elements in $\Es$, see \cref{tab:framework}.
 \begin{table}[!ht]
 	\caption{Entities and operations on $\Ce$ and their counterparts on $\Es$. CholU stands for the upper triangular factor of the Cholesky decomposition.}\label{tab:framework}
	\def\arraystretch{1.5}
	\centering
	\small{\begin{tabular}{ll}
		standard&block\\ \hline \hline
		$\mathbb{C}$ & $\mathbb{S} \simeq \mathbb{C}^{s\times s}$\\
		
		\quad (commutative $^*$-ring) & \quad (noncommutative $^*$-algebra over $\Ce$)\\
		
		$\Er^{+}$ & $\mathbb{S^+}$\ldots upper-$\Delta$ with positive diag. entries\\
		
		\quad (multiplicative topological group)&\quad (multiplicative topological group)\\
		
		$\Er^{+}_0$ & $\mathbb{S}^+_0$\ldots upper-$\Delta$ with nonnegative diag. entries\\
		
		$0$ & singular ${s\times s}$ matrix\\
		
		$1$ & $I\in\Es$\\  \hline
		
		$a,b \in \mathbb{C}$ & $A,B \in \mathbb{S}$\\
		
		$|a| \equiv \sqrt{a^*a}\in\Er_0^+ $ & $|A| = \sqrt{A^*A} \equiv \text{cholU}(A^*A)\in\mathbb{S}_0^+ $\\
				
		$|a|\in \Er^{+} \Longleftrightarrow a\neq 0 $ & $|A|\in \mathbb{S}^{+} \Longleftrightarrow A \text{ nonsingular} $\\
		
		$|a\,|b|| = |a|\,|b|$ &$|A\,|B|| = |A|\,|B|$ \\
		
		\hline
		
		$\bm{x},\bm{y} \in \mathbb{C}^{n}$ & $\bm{X},\bm{Y} \in \mathbb{S}^n\simeq\mathbb{C}^{ns\times s}$\\
		
		$\sxy{\bm{x}}{\bm{y}}\equiv \bm{y}^*\bm{x} \in \mathbb{C}$&$\ssxy{\bm{X}}{\bm{Y}} \equiv \bm{Y}^*\bm{X}\in \mathbb{S}$\\
		
		$\sxy{\bm{x}}{\bm{y}}={\sxy{\bm{y}}{\bm{x}}}^*$&$\ssxy{\bm{X}}{\bm{Y}}=\ssxy{\bm{Y}}{\bm{X}}^*$\\
		
		$\sxy{\bm{x}a}{\bm{y}}=\sxy{\bm{x}}{\bm{y}}a$&$\ssxy{\bm{X}A}{\bm{Y}}=\ssxy{\bm{X}}{\bm{Y}}A$\\
		$\sxy{\bm{x}}{\bm{y}a}=a^*\sxy{\bm{x}}{\bm{y}}$&$\ssxy{\bm{X}}{\bm{Y}A}=A^*\ssxy{\bm{X}}{\bm{Y}}$\\
		
		$\|\bm{x}\| \equiv \sqrt{\sxy{\bm{x}}{\bm{x}}}\in\Er_0^+$ & $\no{\bm{X}} \equiv \sqrt{\ssxy{\bm{X}}{\bm{X}}}\in\mathbb{S}_0^+$ \\
		
		$\|\bm{x}\,|a|\| = \|\bm{x}\|\,|a|$  & $\no{\bm{X}\,|A|} = \no{\bm{X}}\,|A|$ \\
		
		\hline
		
		$I_n\in\Ce^{n\times n}$\ldots identity matrix& $\bcal{I}_n\in\Es^{n\times n}\simeq\mathbb{C}^{ns\times ns}$\ldots block identity matrix\\
		
		$\bm{e}_k$\ldots $k$th column of $I$ & $\bm{E}_k$\ldots$k$th block column of $\bcal{I}$\\
		\hline
					
		$\{\bm{v}_1,\ldots, \bm{v}_n\}$ \ linearly independent&$\{\bm{V}_1,\ldots,\bm{V}_n\}$  \ linearly independent\\
		
		$\|\sum_{i=1}^{n} \bm{v}_ic_i\| = 0 \ \ \Rightarrow \ \  c_i=0\ \forall i$&$\no{\sum_{i=1}^{n} \bm{V}_iC_i} \ \ \text{singular} \ \ \Rightarrow \ \  C_i \ \ \text{singular}\ \forall i
		$\\
		
		\hline
				
		$\{\bm{w}_1,\ldots, \bm{w}_n\}$ \ orthonormal basis of $\Ce^n$ &$\{\bm{W}_1,\ldots, \bm{W}_n\}$ \ orthonormal basis of $\Es^n$\\
		
		$\sxy{\bm{w}_i}{\bm{w}_j} = \delta_{ij}$&$\ssxy{\bm{W}_i}{\bm{W}_j} = \delta_{ij} I$\\
		
		$\bm{x} = \sum_{k=1}^n \bm{w}_k\,\sxy{\bm{x}}{\bm{w}_k}$, \ $\forall \bm{x}\in\Ce^n$ & $\bm{X} = \sum_{k=1}^n \bm{W}_k\,\ssxy{\bm{X}}{\bm{W}_k}$, \ $\forall \bm{X}\in\Es^n$\\
	\end{tabular}}
\end{table}
\begin{notation}\label{notation.caligraphic}
In order to keep our notation clear, we denote matrices
over $\Es$ with caligraphic letters, vectors over $\Es$ by
bold capital letters, and elements of $\Es$ by unbolded capital letters. Vectors over
$\Ce$ are denoted by lower-case bold letters and 
scalar elements of $\Ce$ are 
denoted by lower-case unbolded letters.
\end{notation}

In order to be able to investigate convergence behavior, we generalize the ordering of nonnegative real numbers, denoted by $\Er_0^+$, to that of upper triangular complex matrices with real nonnegative diagonal entries, denoted by $\Es_0^+$, as follows:
\begin{align}\label{eq:gen_loewner}
|A|\prec |B| & \ \Longleftrightarrow \  A^*A\overset{\text{{\tiny Loewner}}}{\prec} B^*B,\ \ \mbox{and}\ \  
|A|\preceq |B| &\ \Longleftrightarrow \ A^*A\overset{\text{{\tiny Loewner}}}{\preceq} B^*B,
\end{align}
where $\overset{\text{{\tiny Loewner}}}{\prec}$ and $\overset{\text{{\tiny Loewner}}}{\preceq}$ is the Loewner (partial) ordering of Hermitian matrices.\footnote{Defined by the relation
	$A\overset{\text{{\tiny Loewner}}}{\prec} B  \iff B-A$ is Hermitian positive definite and $A\overset{\text{{\tiny Loewner}}}{\preceq} B  \iff B-A$ is Hermitian positive semi-definite; see \cite{Loewner1934Uber}.} Most of the results presented in the upcoming sections can be formulated also in terms of the decay of the ``squares'', i.e., the traditional Loewner order. To make the results formally as similar as possible to the classical results for standard GMRES, we have chosen to use the non-increasing order of the norms rather than their squares. 

\subsection{Arnoldi-based methods in the new framework}\label{sec:block_Arnoldi}
Every Krylov subspace method relies on construction of a well-conditioned (ideally orthonormal) basis of the Krylov subspace. For standard GMRES, this orthonormal basis is computed using the Arnoldi algorithm. For block GMRES, where our aim is to compute an orthonormal block basis, see the last part of \Cref{tab:framework}, we use the block Arnoldi algorithm, iteratively producing orthonormal block Arnoldi vectors $\bm{V}_1,\ldots,\bm{V}_k$, $k=1,\ldots,n,$ spanning the Krylov subspace \cref{eq:krylov_subspace}.  If we assume no premature breakdown, it can be carried out for $n$ iterations to produce a basis for the space $\Es^{n}$.   In \Cref{alg:arnoldi}, following \cite{Frommer2017Block}, we provide a pseudocode using the notation introduced in \Cref{tab:framework}.
\begin{algorithm}[t]
	\caption{Block Arnoldi algorithm (run to completion)}
	\label{alg:arnoldi}
	\begin{algorithmic}
		\REQUIRE{$\bcal{A}\in\mathbb{S}^{n\times n}$, $\bm{B}\in\mathbb{S}^{n}$}
		\STATE{$\widetilde{\bm{V}}=\bm{B}$}
		\FOR{$j=1$ to $n$}
		\STATE{$H_{j,j-1} = \no{\widetilde{\bm{V}}}$, \, $\bm{V}_{j} = \widetilde{\bm{V}}H_{j,j-1}^{-1}$}
		\STATE{$\widetilde{\bm{V}} = \bcal{A}\bm{V}_j$}
		\FOR{$i=1$ to $j$}
		\STATE{$H_{i,j} = \ssxy{\widetilde{\bm{V}}}{\bm{V}_i}$}
 		\STATE{$\widetilde{\bm{V}} = \widetilde{\bm{V}}-\bm{V}_iH_{i,j}$}
		\ENDFOR
		\ENDFOR
		\RETURN $\bcal{H}= (H_{i,j})_{i,j=1}^{n}\in\Es^{n\times n}$, $\bcal{V} = \begin{bmatrix}\bm{V}_1&\cdots&\bm{V}_n\end{bmatrix}\in\Es^{n\times n}$
	\end{algorithmic}
\end{algorithm}
There are various strategies for how to handle rank-deficient Arnoldi vectors; see, e.g., \cite{frommer.deflCG.2014,freund.blockqmr.defl.1997,soodhalter.bminres.2015}. Hereafter, we will however assume that there is no breakdown in \Cref{alg:arnoldi}.

Taken to the $n$th iteration, the block Arnoldi algorithm yields the block Arnoldi relation
\begin{equation}\label{eq:block_Arnoldi_mat}
\bcal{A}\bcal{V} = \bcal{V}\bcal{H}, \qquad \bcal{H}\in\Es^{n\times n}\ \mbox{block upper Hessenberg},
\end{equation}
with $\bcal{V}$ having as columns an orthonormal basis of $\Es^{n}$, and the relation \cref{eq:block_Arnoldi_mat} represents a full orthogonal Hessenberg factorization of the matrix $\bcal{A}$. Since the blocks on the subdiagonal of the block Hessenberg matrix $\bcal{H}$ are from $\Es^+$, our definition of $\Es^+$ assures that the matrix $\bcal{H}$ has exactly $s$ subdiagonals and the entries of the last one are positive. Therefore, the block Arnoldi algorithm formulated through the new framework computes the same Arnoldi decomposition as the block Arnoldi algorithm performed in a conventional way; see, e.g., \cite[Section 6.12]{Saad2003Iterative}. Terminating \Cref{alg:arnoldi} after iteration $k$ produces a basis of the Krylov subspace \cref{eq:krylov_subspace}. 

Using the block Arnoldi method (blArnoldi) for eigenvalue problems, the eigenvalues of $\bcal{A}$ are in each step $k$ approximated by the Ritz values, i.e., the $sk$ eigenvalues of the $k$th principal submatrix $\bcal{H}^{(k)}\in\Es^{k\times k}$ of $\bcal{H}$.\footnote{By the eigenvalue of a matrix $\bcal{A}$ we mean the standard eigenvalue, i.e., $\lambda\in\Ce$ satisfying $\bcal{A}\bm{v} = \lambda \bm{v}$ for some nonzero $\bm{v}\in\Ce^{ns}$. We note that considering the linear system over $\Es$ induces a notion of a block eigenvalue decomposition, i.e., there exist $\boldsymbol{\Lambda}\in\Es$ such that $\bcal{A}\bm{V} = \bm{V}\boldsymbol{\Lambda}$ for some $\bm{V}\in\Es^{n}$ with linearly independent columns.  We do not treat such a notion in this paper as it is highly technical, due in part to the non-commutitivity of multiplication in $\Es$.  Furthermore, as we show in \Cref{sec:companion}, the correct notion of block eigenvalues in this setting is that of \emph{block solvents} of specified $\lambda$-matrices.} For systems of linear algebraic equations, the approximate block solutions $\bm{X}_k$ are constructed as a linear combination of the computed orthonormal basis of $\mathcal{K}_k(\bcal{A},\bm{B})$, i.e.,
\begin{equation}
\bm{X}_k = \begin{bmatrix}\bm{V}_1&\cdots&\bm{V}_k\end{bmatrix}\bm{Y}_k,
\end{equation}
where the block vector $\bm{Y}_k\in\Es^k$ is obtained from small projected problems. We analyze two different methods: block FOM (blFOM) and block GMRES (blGMRES). The blFOM method is a Galerkin method, keeping the individual residuals in each step orthogonal to all previous, i.e.,
\begin{equation}\label{eq:FOM}
\bm{Y}_k^F = \left(\bcal{H}^{(k)}\right)^{-1} \bm{E}_1\no{\bm{R}_0}. 
\end{equation}
The blGMRES method minimizes the Euclidean norm of each of the individual residuals\footnote{The application of the Moore-Penrose pseudoinverse to each column individually gives the least-squares solution for each right-hand side, and applying it to each column individually is equivalent to applying it to the block right-hand side. Further, minimizing the Euclidean norm of individual residuals is also equivalent to minimizing the Frobenius norm of the block residual. This was utilized, for example, in the analysis of blGMRES in \cite{Simoncini1996Convergence}.}, i.e.,
\begin{equation}\label{eq:blGMRES}
\bm{Y}^G_k=\left(\underline{\bcal{H}}^{(k)}\right)^\dagger\bm{E}_1\no{\bm{R}_0}, 
\end{equation}
where $\underline{\bcal{H}}^{(k)}\in\Es^{(k+1)\times k}$ is the upper-left $(k+1)\times k$ block of $\bcal{H}$ and $\cdot^\dagger$ denotes here and hereafter the Moore-Penrose pseudoinverse.

While the blGMRES solution always exists, the blFOM solution is not defined when the matrix $\bcal{H}^{(k)}$ is singular.  In this case, we will consider the generalized blFOM solution and define
\begin{equation}\label{eq:gen_FOM}
\bm{Y}_k^F = \operatorname{arg\,min}_{\bm{Y}} \left\{\|\bm{E_k}^T\bm{Y}\|_F \, ; \ \bm{Y} = \left(\bcal{H}^{(k)}\right)^- \bm{E}_1\no{\bm{R}_0} \right\},
\end{equation}
where $\cdot^-$ denotes a generalized inverse; see also \cite{Soodhalter2017Stagnation}.
Since the true and the generalized FOM solution coincide for nonsingular $\bcal{H}^{(k)}$, we denote both by $\bm{X}_k^F$ to simplify the notation.

\section{Admissible convergence behavior of blGMRES}\label{sec:admissible_conv}
In this section, we provide a block definition of admissible convergence behavior of the norm-minimizing method by generalizing some of the well-known relations between the residuals of the norm-minimizing and the Galerkin (residual orthogonalizing) method.

\subsection{Block Givens transformation}\label{sec:Givens}
Each step of standard GMRES requires computation of one new elementary rotation to eliminate the $k$th subdiagonal entry of $\bcal{H}$. These rotations facilitate analysis of GMRES convergence and its relationship to FOM \cite[sec. 6.5.7]{Saad2003Iterative}. For blGMRES, a similar analysis is available using the product of Householder transformations \cite{Soodhalter2017Stagnation}, but a block analog to Givens rotations provides additional and clearer results. 

For the elimination in blGMRES, $s^2$ standard elementary Givens rotations are needed \cite[Section 6.12]{Saad2003Iterative}. A product of (elementary) Givens rotations is an orthogonal transformation but loses the properties of rotation, except for the very special case when all the individual rotations commute; see \cite{Bjoerck1996Numerical,Golub2013Matrix,Tebbens2011Analyza}. Therefore, the \emph{block Givens transformation} will not be a true generalization of a Givens rotation, but rather a block representation of the product of the $s^2$ standard elementary Givens rotations.\footnote{We emphasize that here we are concerned with mathematical properties of such block Givens transformations. In practical computations, the individual Givens sines and cosines are stored, and the product is rarely computed explicitly.}\,\footnote{When $s$ is large, the subdiagonal entries of the block Hessenberg matrix are eliminated using Householder reflections. There are two main possible generalizations of Householder reflections to the block case. The first, see \cite{Schreiber1988Block}, preserves the properties of a reflection, but is only able to eliminate the subdiagonal block. The second one, see \cite{Schreiber1989storage}, is able to eliminate all subdiagonal entries and is a block representation of a product of $s$ standard elementary Householder reflections, but does not have properties of a reflection. In any case, they are of little use when generalizing relations from the standard case, where Givens rotations always are performed.}

We follow the idea of \cite{HalleckBlock}. Assume first a general orthogonal transformation $\bcal{Q}$ that eliminates an entry of a block vector, i.e.,
\begin{equation}
\bm{V} = \begin{bmatrix}V_{1}\\V_{2}\end{bmatrix}\longrightarrow \bcal{Q}\bm{V}=\begin{bmatrix}\widetilde{V}\\0\end{bmatrix}, \label{eq:elimination}
\end{equation}
where $V_1, V_2\in\Es$ and $|\widetilde{V}| = \no{\bm{V}}$. 
If $V_2$ is invertible, then the unitary matrix $\bcal{Q}$ eliminating the block $V_2$ can be constructed as
\begin{equation}
\bcal{Q} = \begin{bmatrix}
\bar{C}&\bar{S}\\
-S&C\\
\end{bmatrix} 
:=
\begin{bmatrix}
XZ^*&X\\
-Y&YZ
\end{bmatrix}, \label{eq:givens_mat}
\end{equation}
where 
\begin{align}
Z = V_1V_2^{-1},\qquad
X^*X =(I+Z^*Z)^{-1},\qquad \mbox{and}\qquad 
Y^*Y =(I+ZZ^*)^{-1},
\end{align}
which can be verified by simple computation.

Note that there is freedom in the choice of $X$ and $Y$, since they can be arbitrary right factors of the matrices $(I+Z^*Z)^{-1}$ and $(I+ZZ^*)^{-1}$, respectively. This allows us also to control the non-zero pattern of the matrix $\bcal{Q}$. If we take
\begin{align}\label{eq:givens_choice}
X = \left(\sqrt{(I+Z^*Z)}\right)^{-*}\qquad \mbox{and}\qquad
Y = \sqrt{(I+ZZ^*)^{-1}}, 
\end{align}
then, due to our definition of $\sqrt{\,\cdot\,}$\,, the matrix $\bcal{Q}$ in \cref{eq:givens_mat} becomes $(2s+1)$-diagonal. In this case, we can see $\bcal{Q}$ as a product of individual standard Givens rotations applied to $\bm{V}$ to eliminate the subdiagonal entries. If $V_1,V_2\in\Ce$, the choice \cref{eq:givens_choice} will lead to the standard elementary Givens rotation with $\bar{S}$ and $\bar{C}$ being the complex conjugates of $C$ and $S$, respectively. This is however generally not the case for $V_1,V_2\in\Es$, i.e., $\bar{C}\neq C^*$ and $\bar{S}\neq S^*$. The following proposition explains why the choice \cref{eq:givens_choice} is important for the solution of \cref{eq:blGMRES} by factorizing the block upper Hessenberg matrix $\bcal{H}$ from \cref{eq:block_Arnoldi_mat}.
\begin{proposition}\label{th:givens}
	Let $\bcal{H}\in\Es^{n\times n}$ be an upper block Hessenberg matrix with blocks $H_{i,j}$, and let further the subdiagonal blocks be upper triangular with positive diagonal entries, i.e., $H_{k+1,k}\in\Es^+$, $k = 1,\ldots, n-1$. Define, for $k = 1,\ldots,n-1$, the block Givens transformations $\bcal{G}^{(k)}$ as
	\begin{equation}
	\bcal{G}^{(k)} = \begin{bmatrix}\bcal{I}_{k-1}\\&\bar{C}_k&\bar{S}_k\\&-S_k&C_k\\&&&\bcal{I}_{n-k-1}\end{bmatrix}, \label{eq:givens_big}
	\end{equation} 
	where
	\begin{equation}
	\begin{bmatrix}
	\bar{C_k}&\bar{S_k}\\-S_k&C_k\\
	\end{bmatrix} =
	\begin{bmatrix}
	\sqrt{(I+Z_k^*Z_k)}^{-*}Z_k^*&\sqrt{(I+Z_k^*Z_k)}^{-*}\\
	-\sqrt{(I+Z_kZ_k^*)^{-1}}&\sqrt{(I+Z_kZ_k^*)^{-1}}Z_k
	\end{bmatrix}, \quad Z_k = H_{k,k}^{(k-1)}H_{k+1,k}^{-1},
	\end{equation}
	with $\bcal{H}^{(0)} = \bcal{H}$ and $\bcal{H}^{(k)} = \bcal{G}^{(k)}\bcal{H}^{(k-1)}$. Define further the unitary matrix $\bcal{Q}$ as
	\begin{equation}
	\bcal{Q} = \bcal{G}^{(n-1)}\, \bcal{G}^{(n-2)}\cdots \bcal{G}^{(1)}.
	\end{equation}
	Then $\bcal{QH}$ is a block upper triangular matrix, with the first $n-1$ block diagonal entries in $\Es^+$.
\end{proposition}
\begin{proof} 
	The unitarity of $\bcal{Q}$ and the upper block triangular form of $\bcal{QH}$  follows directly from the construction and discussion above. Further, observe that
	\begin{align}
	\begin{bmatrix}
	\bar{C_k}&\bar{S_k}\\-S_k&C_k\\
	\end{bmatrix}
	\begin{bmatrix}H_{k,k}^{(k-1)}\\H_{k+1,j}\end{bmatrix}
	=\begin{bmatrix}\Xi\\0\end{bmatrix} &\Longrightarrow \begin{bmatrix}H_{k,k}^{(k-1)}\\H_{k+1,k}\end{bmatrix}
	=\begin{bmatrix}
	\bar{C}_k^*&-S_k^*\\\bar{S}_k^*&C_k^*\\
	\end{bmatrix}\begin{bmatrix}\Xi\\0\end{bmatrix}\\
	&\Longrightarrow H_{k+1,k} = \bar{S}_k^*\Xi\\
	&\Longrightarrow \Xi = \bar{S}_k^{-*}H_{k+1,k}.
	\end{align}
	Since $\Es^+$ is a multiplicative group and both $\bar{S}_k^{-*}\in\Es^+$ and $H_{k+1,k}\in\Es^+$, also $\Xi\in\Es^+$, which yields the desired statement.
\end{proof}

\Cref{th:givens} shows that the $n-1$ block Givens transformations $\bcal{G}_1,\ldots,\bcal{G}_{n-1}$ defined in \cref{eq:givens_big}, with additional normalization of the last block entry of $\bcal{H}^{(n-1)}$, enable reduction of the Hessenberg matrix $\bcal{H}$ to the upper triangular form. 

\subsection{Residual norms and the peak-plateau relation}\label{sec:peak_plateau}
In this section, we use the block Givens transformations to generalize some of the well-know results about GMRES and FOM residuals. Note that for the generalized blFOM solution \cref{eq:gen_FOM}, the analysis only addresses the component of the true residual that satisfies the desired Galerkin condition, i.e., $\bm{R}_k^F :=\bm{V}_{k+1}\bm{V}_{k+1}^*\left(\bm{B} - \bcal{A}\bm{X}_k^F\right).$ Since the blGMRES solution always exists, it is defined in the conventional way as $\bm{R}_k^G :=\bm{B} - \bcal{A}\bm{X}_k^G$.

Let $\bm{R}_0= \bm{V}_1\no{\bm{R}_0}$ be the initial residual. Note that since we assume no premature breakdown in \Cref{alg:arnoldi}, it holds that $\no{\bm{R}^{G}_{k}}$ is invertible for $k=0,\ldots,n-1$. We follow \cite[sec. 6.5.7]{Saad2003Iterative}. Since the $k$th Givens rotation modifies the right-hand side of the projected problem as 
\begin{equation}
\begin{bmatrix}
\bar{C}_k&\bar{S}_k\\
-S_k&C_k\\
\end{bmatrix}\begin{bmatrix}\no{\bm{R}^{G}_{k-1}}\\0\end{bmatrix}=\begin{bmatrix}*\\-S_k\no{\bm{R}_{k-1}^{G}}\end{bmatrix}
\end{equation}
and since $S_k\in\Es^+$, we get that the blGMRES residual satisfies
\begin{equation}
\no{\bm{R}^{G}_k} = S_k\no{\bm{R}^{G}_{k-1}}\,, \qquad\mbox{which implies}\qquad \no{\bm{R}^{G}_k} = S_k\,\cdots S_1\no{\bm{R}_0}.
\end{equation}
Using \cref{th:givens}, we obtain that the (generalized) blFOM  residual satisfies
\begin{align}\label{eq:fom_res}
\no{\bm{R}^{F}_k} &= |H_{k+1,k}\bm{E}_k^T\bm{Y}_k^F|\\
&= \left\vert H_{k+1,k}\left(H^{(k-1)}_{k,k}\right)^\dagger\right\vert\,\no{\bm{R}^{G}_{k-1}}\\
&= |C_k^\dagger S_k|\,\no{\bm{R}^{G}_{k-1}}= |C_k^\dagger|\,S_k\,\no{\bm{R}^{G}_{k-1}} =|C_k^\dagger|\,\no{\bm{R}^{G}_k},
\end{align}
where $C_k^\dagger$ becomes $C_k^{-1}$, whenever the $k$th principal submatrix of $\bcal{H}$, and therefore also $H^{(k-1)}_{k,k}$, is invertible; cf. \cref{eq:gen_FOM}.

Since blGMRES and blFOM form a norm-minimizing/Galerkin pair, one expects that they satisfy some form of \emph{peak-plateau relation}, see, e.g., \cite[sec. 6.5.7]{Saad2003Iterative} or \cite{Cullum1996Relations}. The following proposition shows that this is indeed the case.

\begin{proposition}
	The residuals of blGMRES and blFOM satisfy
	\begin{equation}\label{eq:peak_plateau}
		\ssxy{\bm{R}_k^F}{\bm{R}_k^F}^\dagger = \ssxy{\bm{R}_k^G}{\bm{R}_k^G}^{-1}-\ssxy{\bm{R}_{k-1}^G}{\bm{R}_{k-1}^G}^{-1}
	\end{equation}
	and 
	\begin{equation}
	\ssxy{\bm{R}_k^G}{\bm{R}_k^G}^{-1} = \sum_{i=0}^k\ssxy{\bm{R}_i^F}{\bm{R}_i^F}^\dagger.
	\end{equation}
\end{proposition}
\begin{proof}	
Using \cref{eq:fom_res}, we have
\begin{align}\label{eq:peak_plateau_prep}
\ssxy{\bm{R}_k^F}{\bm{R}_k^F}&=\no{\bm{R}_k^G}^*\,C_k^{\dagger*}\,C_k^{\dagger}\,\no{\bm{R}_k^G}=\no{\bm{R}_k^G}^*\left(I-S_kS_k^*\right)^\dagger\no{\bm{R}_k^G} \\
&=\no{\bm{R}_k^G}^*\left(I-\no{\bm{R}^{G}_k}\,\no{\bm{R}^{G}_{k-1}}^{-1}\no{\bm{R}^{G}_{k-1}}^{-*}\no{\bm{R}^{G}_k}^*\right)^\dagger\no{\bm{R}_k^G}. 
\end{align}
By taking the pseudoinverse, we obtain
\begin{align}
\ssxy{\bm{R}_k^F}{\bm{R}_k^F}^\dagger &= \no{\bm{R}_k^G}^{-1}\left(I-\no{\bm{R}^{G}_k}\,\no{\bm{R}^{G}_{k-1}}^{-1}\no{\bm{R}^{G}_{k-1}}^{-*}\no{\bm{R}^{G}_k}^*\right)\no{\bm{R}_k^G}^{-*}\\
&= \no{\bm{R}_k^G}^{-1}\no{\bm{R}_k^G}^{-*}-\no{\bm{R}_{k-1}^G}^{-1}\no{\bm{R}_{k-1}^G}^{-*}\\
&= \ssxy{\bm{R}_k^G}{\bm{R}_k^G}^{-1}-\ssxy{\bm{R}_{k-1}^G}{\bm{R}_{k-1}^G}^{-1}, 
\end{align}
which gives \cref{eq:peak_plateau}. Applying relation \cref{eq:peak_plateau} recursively gives the second statement.
\end{proof}
Relation \cref{eq:peak_plateau} can be viewed as a generalization of the peak-plateau relation, to which it reduces when $s=1$. 
For other relations between blGMRES and blFOM, see also \cite{Soodhalter2017Stagnation}.

Also note that $\ssxy{\bm{R}_k^F}{\bm{R}_k^F}$ is singular if and only if $\no{\bm{R}_k^F}$ is singular, i.e., if either the standard FOM iterate does not exist for some of the right-hand sides, or if the individual FOM residuals are linearly dependent. This corresponds to the situation when the residual update in blGMRES is of rank smaller than $s$, see \cite[p. 173]{Soodhalter2017Stagnation}. Such situation will be of interest also in \Cref{sec:Arnoldi_vs_GMRES}.

\subsection{Admissible convergence behavior of blGMRES}
The results of the previous section have some nontrivial consequences for the convergence behavior of blGMRES.  
\begin{theorem}\label{th:admissible_conv}
	The blGMRES residuals satisfy
	\begin{equation}\label{eq:admissible}
	\no{\bm{R}_0}\succeq \no{\bm{R}^G_1} \succeq \cdots \succeq \no{\bm{R}_{n-1}^G}\succ 0.
	\end{equation}
\end{theorem}
\begin{proof}
	Since $\ssxy{\bm{R}_k^F}{\bm{R}_k^F}$ is a positive semidefinite matrix, relation~\cref{eq:peak_plateau} directly implies that for $k=1,\ldots,n$
	\begin{equation}
	\ssxy{\bm{R}_k^G}{\bm{R}_k^G} \overset{\text{{\tiny Loewner}}}{\preceq} \ssxy{\bm{R}_{k-1}^G}{\bm{R}_{k-1}^G}. 
	\end{equation}
	Using the generalization of the Loewner ordering from \Cref{sec:preliminaries} we can write that
	\begin{equation}\label{eq:gmres_residual_decay}
	\no{\bm{R}_k^G} \preceq \no{\bm{R}_{k-1}^G}. 
	\end{equation}
	Applying this relation recursively, we obtain \cref{eq:admissible}.
\end{proof}

We call the sequence $\no{\bm{R}_0},\no{\bm{R}_1},\ldots,\no{\bm{R}_{n-1}}$ of elements in $\Es^+$ an \emph{admissible convergence behavior} of blGMRES, if it satisfies \cref{eq:admissible}. 

\begin{remark}
	Using $\|\bm{R}_k^Ge_j\|^2=\sxy{\bm{R}_k^Ge_j}{\bm{R}_k^Ge_j} = e_j^T\ssxy{\bm{R}_k^G}{\bm{R}_k^G}e_j$, $j= 1,\ldots,s$, relation \cref{eq:admissible} trivially implies monotonic convergence of the size of the individual residuals. However, relation \cref{eq:admissible}  is generally stronger, because it takes into account the inter-residual relationships, expressed by the off-diagonal entries of the matrices $\ssxy{\bm{R}_k^G}{\bm{R}_k^G}$, $k=0,\ldots,n-1$. If two individual initial residuals (of the same norm) are almost linearly dependent, one cannot expect radically different convergence behaviors for each right-hand side. We will demonstrate this with the following example. Let the initial residuals be almost linearly dependent:
	\begin{equation}
	\ssxy{\bm{R}_0^G}{\bm{R}_0^G} := \begin{bmatrix}1 & 1-\varepsilon\\1-\varepsilon & 1\end{bmatrix}, \quad \varepsilon = 0.01,
	\end{equation}
	and let the size of the first residual be decreased to $\sqrt{\varepsilon}$ and of the second one to $\sqrt{1-\varepsilon}$:
	\begin{equation}
	\ssxy{\bm{R}_1^G}{\bm{R}_1^G} := \begin{bmatrix}\varepsilon & p\\p & 1-\varepsilon\end{bmatrix}, \quad p \text{ unknown}.
	\end{equation}
	Then there is no $p$ such that $	\ssxy{\bm{R}_0^G}{\bm{R}_0^G}\overset{\text{{\tiny Loewner}}}{\succeq} \ssxy{\bm{R}_1^G}{\bm{R}_1^G} \overset{\text{{\tiny Loewner}}}{\succ} 0$,
	and therefore such convergence behavior cannot be exhibited by blGMRES. Conversely, if two initial residuals are orthogonal, \cref{eq:gmres_residual_decay} does not give any further restriction on the convergence curve of the individual residuals.
\end{remark}

\begin{remark}Note that the result of \cref{th:admissible_conv} is also very intuitive in the following sense: Any non-increasing sequence (in the Loewner sense) can be generated through the norms of the orthogonal projections of a (block) vector onto a sequence of embedded subspaces. And vice versa, norms of the orthogonal projections of a (block) vector to any sequence of embedded subspaces will generate a non-increasing sequence (in the Loewner sense). Since orthogonal projections of the residuals to the residual Krylov subspace are the very essence of the GMRES method, the Loewner ordering of its residual sizes is its inherent property.
\end{remark}

The Loewner ordering is also natural to consider when block Krylov methods are formulated via polynomial approximation of the inverse of $\bcal{A}$. The polynomial with matrix-valued coefficients applied to the right-hand side $\bm{B}$ to obtain the residual is optimal in the sense of the following theorem.  We use the concept of matrix-valued polynomials, which we define in \cref{eq:lambda_matrix} and elaborate on in \Cref{sec:poly_mat}, to give the proper generalization of the GMRES residual polynomial minimization.
\begin{theorem}
Denote by $\mathcal{P}_k$ the set of all $\lambda$-matrices $P$ of degree at most $k$, with coefficients in $\Es$,  and satisfying $P(0) = I$, see \Cref{sec:poly_mat}.  The blGMRES residual norms satisfy
	\begin{equation}\label{eq:res_poly_block}
	\no{\bm{R}_k^G}= \underset{P\in\mathcal{P}_k}{\operatorname{argmin}} \no{P(\bcal{A})\circ \bm{B}},
	\end{equation}
	where $\circ$ is as in \cref{eq:circ_operation} and the minimum is considered in the generalized Loewner sense \cref{eq:gen_loewner}.
\end{theorem}
\begin{proof}
From the definition of the blGMRES iterate it follows that $\bm{R}_k^G = \widehat{P}_k(\bcal{A})\circ\bm{B}$ for some $\widehat{P}_k\in\mathcal{P}_k$.
Therefore, it suffices to show that by changing the polynomial, one can only make the norm of the resulting block vector larger. 
Any $\lambda$-matrix in $\mathcal{P}_k$ can be written as $\widehat{P}_k(\lambda)+\lambda Q_{k-1}(\lambda)$, where $Q_{k-1}$ is a $\lambda$-matrix of degree at most $k-1$. Therefore, we can write
	\begin{align}\label{eq:shift_poly}
	&\ssxy{[\widehat{P}_k(\bcal{A})+\bcal{A}Q_{k-1}(\bcal{A})]\circ \bm{B}}{[\widehat{P}_k(\bcal{A})+\bcal{A}Q_{k-1}(\bcal{A})]\circ \bm{B}}\\
	&\hspace{.5cm} = \ssxy{\bcal{A}Q_{k-1}(\bcal{A})\circ \bm{B}}{\bcal{A}Q_{k-1}(\bcal{A})\circ \bm{B}} + \ssxy{\widehat{P}_k(\bcal{A})\circ \bm{B}}{\widehat{P}_k(\bcal{A})\circ \bm{B}}\\
	&\hspace{1cm} + \ssxy{\widehat{P}_k(\bcal{A})\circ \bm{B}}{\bcal{A}Q_{k-1}(\bcal{A})\circ \bm{B}}+ \ssxy{\bcal{A}Q_{k-1}(\bcal{A})\circ \bm{B}}{\widehat{P}_k(\bcal{A})\circ \bm{B}}.
	\end{align}
	Further, we observe that $\bcal{A}Q_{k-1}(\bcal{A})\circ \bm{B}\in\bcal{A}\mathcal{K}_k(\bcal{A},\bm{B})$. Since the residual $\bm{R}^G_k$ is orthogonal to this space, we obtain
	\begin{equation}
	0 = \ssxy{\widehat{P}_k(\bcal{A})\circ \bm{B}}{\bcal{A}Q_{k-1}(\bcal{A})\circ \bm{B}} = \ssxy{\bcal{A}Q_{k-1}(\bcal{A})\circ \bm{B}}{\widehat{P}_k(\bcal{A})\circ \bm{B}}^*.
	\end{equation}
	With this, \cref{eq:shift_poly} becomes
	\begin{align}
	&\ssxy{[\widehat{P}_k(\bcal{A})+\bcal{A}Q_{k-1}(\bcal{A})]\circ \bm{B}}{[\widehat{P}_k(\bcal{A})+\bcal{A}Q_{k-1}(\bcal{A})]\circ \bm{B}}\\
	&\hspace{1cm} = \overbrace{\ssxy{\bcal{A}Q_{k-1}(\bcal{A})\circ \bm{B}}{\bcal{A}Q_{k-1}(\bcal{A})\circ \bm{B}}}^{\succeq 0} + \ssxy{\widehat{P}_k(\bcal{A})\circ \bm{B}}{\widehat{P}_k(\bcal{A})\circ \bm{B}},
	\end{align}
	which finishes the proof.
\end{proof}

In the following section we show that \underline{any} admissible convergence behavior \cref{eq:admissible} is actually attainable by blGMRES. For better readability, we drop the superscript $\cdot^{G}$ in the remainder of this paper.

\section{Prescribing convergence behavior}\label{sec:prescribing_conv}
In this section, we utilize the new framework for block Krylov subspace methods and generalize some of the results of \cite{Arioli1998Krylov,DuintjerTebbens2012Any,Greenbaum1994Matrices,Greenbaum1996Any} for standard Arnoldi and GMRES to the block case. We show that, under moderate conditions, we simultaneously can prescribe the residual convergence of blGMRES($\bcal{A}$,$\bm{B}$) as well as the spectral properties of $\bcal{A}$ and the principal submatrices $\bcal{H}^{(k)}$ of the ultimate upper Hessenberg matrix $\bcal{H}$ produced by blArnoldi($\bcal{A}$,$\bm{B}$). These entities will be prescribed based on the framework introduced in \Cref{sec:framework} and have a slightly different form than in the standard case.

To prescribe the blGMRES residual norms, we require
\begin{equation}\label{eq:residual_match}
\no{\bm{R}_k} = F_k, \quad k = 0,\ldots,n-1,
\end{equation}
for some given sequence $\{F_k\}_{k=0}^{n-1}$, $F_k\in\Es^+$, satisfying the admissibility condition, with
\begin{equation}
F_0\succeq F_1 \succeq F_2 \succeq \cdots \succeq F_{n-1}\succ 0.
\end{equation}

The spectral properties of $\bcal{A}$ are in the standard case prescribed through the similarity to the companion matrix corresponding to the characteristic polynomial defined by these eigenvalues; see \cite{Greenbaum1996Any}. We will show that by considering block companion matrices of the form
\begin{equation}\label{eq:companion_matrix}
\bcal{C} =
\begin{bmatrix}
0& && C_0\\
I&\ddots&&C_1\\
&\ddots&0&\vdots\\
& & I &C_{n-1}
\end{bmatrix},
\end{equation}
we can proceed analogously in the block case.  Note that the matrix $\bcal{C}$ defined in \cref{eq:companion_matrix} is the block companion matrix corresponding to the $\lambda$-matrix
\begin{equation}
M(\lambda) = \lambda^nI - \sum_{k=0}^{n-1}\lambda^kC_k, \ \ \lambda\in\Ce.\label{eq:lambda_matrix_II}
\end{equation}
We will be looking for the matrices $\bcal{A}$ and the right-hand sides $\bm{B}$ annihilating the polynomial $M$, i.e., 
\begin{equation}\label{eq:char_poly} 
M(\bcal{A})\circ \bm{B} = 0,
\end{equation}   
where the operation $\circ$ is as in \cref{eq:circ_operation}. Using $\bcal{A} = \bcal{V}\bcal{H}\bcal{V}^*$ and $\bm{B} = \bm{V}_1\no{\bm{R}_0}$, \cref{eq:char_poly} can be rewritten as
\begin{equation}
M(\bcal{H})\circ\bm{E}_1\no{\bm{B}} = 0.
\end{equation} 
The relation between the spectral properties of $\bcal{A}$ and condition \cref{eq:char_poly} is discussed in more detail in \Cref{sec:companion}.

Due to \cite{DuintjerTebbens2012Any}, it is known that in the standard case, not only the eigenvalues of the matrix $\bcal{A}$ and therefore also the ultimate upper Hessenberg matrix $\bcal{H}$ can be prescribed, but also the eigenvalues of all the principal block submatrices $\bcal{H}^{(k)}$ of this Hessenberg matrix, i.e., the Ritz values. In the block case, spectral properties of the submatrices $\bcal{H}^{(k)}$ will be enforced analogously to the spectral properties of $\bcal{A}$ through
\begin{equation}\label{eq:char_poly_Ritz}
M^{(k)}\left(\bcal{H}^{(k)}\right)\circ \bm{E}_1\no{\bm{B}} = 0, \quad k = 1,\ldots,n-1,
\end{equation} 
where $M^{(k)}$ is defined as
\begin{equation}\label{eq:lambda_matrix_Ritz}
M^{(k)}(\lambda) = \lambda^kI - \sum_{j=0}^{k-1}\lambda^jC_j^{(k)}, \ \ \lambda\in\Ce.
\end{equation}

\subsection{Attaining prescribed spectral properties of \texorpdfstring{{$\bcal{A}$}}{A}}\label{sec:spectral_A}
Defining the Krylov matrix
\begin{equation}
\bcal{K} := \begin{bmatrix}\bm{B}&\bcal{A}\bm{B}&\bcal{A}^2\bm{B}&\cdots&\bcal{A}^{n-1}\bm{B}\end{bmatrix} \in \Es^{n\times n},
\end{equation}
condition \cref{eq:char_poly} is equivalent to
\begin{equation}\label{eq:AK_KC}
\bcal{A}\bcal{K} = \bcal{K}\bcal{C},
\end{equation}
yielding
\begin{equation}\label{eq:A_KCK}
\bcal{A} = \bcal{K}\bcal{C}\bcal{K}^{-1}, \quad \bm{B} = \bcal{K}\bm{E}_1.
\end{equation} 
Consider the unique $QR$-decomposition (with separate diagonal scaling matrix $\bcal{D}$) of the matrix $\bcal{K}$,
\begin{equation}\label{eq:K_VDU}
\bcal{K} = \bcal{V}\bcal{D}\bcal{U},
\end{equation}
with $\bcal{V}$ unitary, $\bcal{D}$ non-singular block diagonal with diagonal entries in $\Es^+$, and $\bcal{U}$ non-singular upper triangular with unit diagonal entries. With this decomposition, equations \cref{eq:A_KCK} become
\begin{equation}\label{eq:A_VUCUV}
\bcal{A} = \bcal{V}\underbrace{\bcal{D}\bcal{U}\bcal{C}\bcal{U}^{-1}\bcal{D}^{-1}}_{\bcal{H}}\bcal{V}^*, \quad \bm{B} = \bcal{V}\bm{E}_1D_{1}.
\end{equation}
Similar to the standard case, see, e.g., \cite[sec. 3]{Meurant2012GMRES}, any product of the form $\bcal{D}\bcal{U}\bcal{C}\bcal{U}^{-1}\bcal{D}^{-1}$ is a block upper Hessenberg matrix with subdiagonal entries in $\Es^+$ and vice versa. Moreover, $\bcal{DU}$ satisfies
\begin{equation}
	\bcal{DU} = \begin{bmatrix}\bm{E}_1D_1&\bcal{H}\bm{E}_1D_1&\cdots&\bcal{H}^{n-1}\bm{E}_1D_1\end{bmatrix},
\end{equation}
where $D_1$ is the first block entry of $\bcal{D}$.\,\footnote{Note that contrary to the standard case, for a given upper Hessenberg matrix $\bcal{H}$, the matrices $\bcal{U}$ and $\bcal{C}$ are not defined uniquely, since they  depend on the choice of $D_1$. For the same reason, the scaling by $\no{\bm{B}}$ cannot be easily omitted in \cref{eq:char_poly_Ritz}.}

Using the decomposition \cref{eq:A_VUCUV}, we provide a complete characterization of matrices and right-hand sides providing prescribed convergence behavior. The convergence behavior of blArnoldi and blGMRES is unitarily invariant; therefore the choice of $\bcal{V}$ plays no role in the analysis. Note that from \cref{eq:A_VUCUV}, we already have that $D_1 = \no{\bm{B}}$. In the following sections, we determines the $(n-1)(n-2)/2$ block entries of $\bcal{U}$ and the last $n-1$ diagonal block entries of $\bcal{D}$ so that the prescribed blGMRES and blArnoldi behavior is met.\footnote{In \cite{Greenbaum1996Any,Arioli1998Krylov}, the characterization of matrices and right-hand sides providing prescribed behavior is obtained through the factorization of the matrix corresponding to the Krylov residual subspaces $\bcal{A}\mathcal{K}_n(\bcal{A},\bm{B})$. This factorization provides alternative formulation of the results presented in the subsequent sections.}

\subsection{Attaining prescribed blArnoldi convergence}\label{sec:spectral_H}
We show that the blArnoldi convergence \cref{eq:char_poly_Ritz} is encoded solely in the upper triangular matrix $\bcal{U}$. 

Using the definition of $M^{(k)}(\lambda)$, condition \cref{eq:char_poly_Ritz} can be rewritten using \cref{eq:lambda_matrix_Ritz} as
\begin{equation}\label{eq:Hk_poly}
\left(\bcal{H}^{(k)}\right)^k\bm{E}_1\no{\bm{B}} = \sum_{j=0}^{k-1}\left(\bcal{H}^{(k)}\right)^j\bm{E}_1\no{\bm{B}}\,C^{(k)}_j, \quad k = 1,\ldots,n-1.
\end{equation}
Since $\bcal{H} = \bcal{D}\bcal{U}\bcal{C}\bcal{U}^{-1}\bcal{D}^{-1}$ and $\bcal{U}^{-1}\bcal{D}^{-1}\bm{E}_1\no{\bm{B}} = \bm{E}_1$, we observe that
\begin{equation}
\left(\bcal{H}^{(k)}\right)^j\bm{E}_1\no{\bm{B}} = \bcal{D}^{(k)}\bcal{U}^{(k)}\bm{E}_{j+1}, \quad j = 0,\ldots,k-1,
\end{equation}
where $\bcal{D}^{(k)}$ and $\bcal{U}^{(k)}$ are the $k$th principal submatrices of $\bcal{D}$ and $\bcal{U}$, respectively.
Substituting into \cref{eq:Hk_poly}, we have, for $k = 1,\ldots,n-1$,
\begin{equation}\label{eq:Hk_poly_1}
\bcal{H}^{(k)}\bcal{D}^{(k)}\bcal{U}^{(k)}\bm{E}_{k} = \bcal{D}^{(k)}\bcal{U}^{(k)}\sum_{j=0}^{k-1}\bm{E}_{j+1}C^{(k)}_j.
\end{equation}
Using the factorization of $\bcal{H}$ and the structure of $\bcal{D}$, $\bcal{U}$, and $\bcal{C}$, the left-hand side of \cref{eq:Hk_poly_1} can be rewritten as
\begin{align}\label{eq:Ritz_I}
\bcal{H}^{(k)}\bcal{D}^{(k)}\bcal{U}^{(k)}\bm{E}_{k} &= \left(\begin{bmatrix}\bcal{I}_k&\bm{0}\end{bmatrix}\bcal{D}\bcal{U}\bcal{C}\bcal{U}^{-1}\bcal{D}^{-1}\begin{bmatrix}\bcal{I}_k\\\bm{0}\end{bmatrix}\right)\bcal{D}^{(k)}\bcal{U}^{(k)}\bm{E}_{k}\\
&= \left(\begin{bmatrix}\bcal{I}_k&\bm{0}\end{bmatrix}\bcal{D}\right)\bcal{U}\bcal{C}\left(\bcal{U}^{-1}\bcal{D}^{-1}\begin{bmatrix}\bcal{I}_k\\\bm{0}\end{bmatrix}\bcal{D}^{(k)}\right)\bcal{U}^{(k)}\bm{E}_{k}\\
&=\bcal{D}^{(k)}\begin{bmatrix}\bcal{I}_k&\bm{0}\end{bmatrix}\bcal{U}\bcal{C}\begin{bmatrix}\left(\bcal{U}^{(k)}\right)^{-1}\\0\end{bmatrix}\bcal{U}^{(k)}\bm{E}_{k}\\
&=\bcal{D}^{(k)}\begin{bmatrix}\bcal{I}_k&\bm{0}\end{bmatrix}\bcal{U}\bcal{C}\bm{E}_{k}\\
&=\bcal{D}^{(k)}\begin{bmatrix}\bcal{I}_k&\bm{0}\end{bmatrix}\bcal{U}\bm{E}_{k+1}\\
&=\bcal{D}^{(k)}\bcal{U}^{(k)}\begin{bmatrix}\left(\bcal{U}^{(k)}\right)^{-1}&\bm{0}\end{bmatrix}\bcal{U}\bm{E}_{k+1}.
\end{align}
Furthermore, using the block inversion formula and the fact that that $\bcal{U}$ is upper triangular with identity matrices on the main diagonal, we have
\begin{equation}\label{eq:Ritz_II}
\begin{bmatrix}\left(\bcal{U}^{(k)}\right)^{-1}&\bm{0}\end{bmatrix}\bcal{U}\bm{E}_{k+1}=\left(-\begin{bmatrix}\bcal{I}_k&\bm{0}\end{bmatrix}\left(\bcal{U}^{(k+1)}\right)^{-1}\right)\bm{E}_{k+1}.
\end{equation}

 Substituting \cref{eq:Ritz_I,eq:Ritz_II} back into \cref{eq:Hk_poly_1} and premultiplying both sides by the inverse of $\bcal{D}^{(k)}\bcal{U}^{(k)}$, we finally obtain
\begin{equation}\label{eq:Hk_poly_2}
-\begin{bmatrix}\bcal{I}_k&\bm{0}\end{bmatrix}\left(\bcal{U}^{(k+1)}\right)^{-1}\bm{E}_{k+1}= \sum_{j=0}^{k-1}\bm{E}_{j+1}C^{(k)}_j, \quad k = 1,\ldots,n-1.
\end{equation}
Conditions \cref{eq:Hk_poly_2} uniquely determine the upper triangular matrix $\bcal{U}$ as
\begin{equation}\label{eq:U_Ritz}
\bcal{U} = \begin{bmatrix}
I&-C_0^{(1)}& -C_0^{(2)}&\cdots&-C_0^{(n-1)}\\
&I&-C_1^{(2)}& \cdots & \vdots\\
& & \ddots & \ddots & \vdots\\
& & & I& -C_{n-2}^{(n-1)}\\
& & & & I\\
\end{bmatrix}^{-1}.
\end{equation}

\subsection{Attaining prescribed blGMRES convergence}\label{sec:Arnoldi_vs_GMRES}
We show that for a given $\bcal{U}$, the prescribed blGMRES convergence \cref{eq:residual_match} can be achieved by a proper choice of the block diagonal matrix $\bcal{D}$, which also defines the subdiagonal entries of $\bcal{H}$. 

Let $\bcal{W}$ be any matrix such that its columns $\bm{W}_1,\ldots,\bm{W}_k$ form an orthonormal basis of $\bcal{A}\mathcal{K}_k(\bcal{A},\bm{B})$, $k=1,\ldots,n$. To satisfy \cref{eq:residual_match}, it has to hold that
\begin{align}\label{eq:G_def}
|\ssxy{\bm{B}}{\bm{W}_k}| &=  \sqrt{\ssxy{F_{k-1}}{F_{k-1}} - \ssxy{F_k}{F_k}} =: G_k, \quad k = 1,\ldots,n-1,\\
|\ssxy{\bm{B}}{\bm{W}_n}| &=  \sqrt{\ssxy{F_{n-1}}{F_{n-1}}} = {F_{n-1}} =: G_n.
\end{align}
Now we relate the columns of $\bcal{W}$ to those of $\bcal{V}$. First, we see that
\begin{align}\label{eq:WH_hat_inv}
\bcal{K}
= 
\bcal{W}
{\small \begin{bmatrix}
	\ssxy{\bm{B}}{\bm{W}_1}&I & & \\
	\ssxy{\bm{B}}{\bm{W}_2}&0&\ddots&\\
	\vdots& &\ddots&I\\
	\ssxy{\bm{B}}{\bm{W}_n} & & &0
	\end{bmatrix}}
\begin{bmatrix}I&\\&{\bcal{R}}\end{bmatrix}
=
 \bcal{W}\bcal{Q}{\small\begin{bmatrix}
	G_1&I & & \\
	G_2&0&\ddots&\\
	\vdots& &\ddots&I\\
	G_n & & &0
	\end{bmatrix}}
\begin{bmatrix}I&\\&\widehat{\bcal{R}}\end{bmatrix},
\end{align} 
where $\bcal{Q}\in\Es^{n\times n}$ is unitary block diagonal and ${\bcal{R}},\widehat{\bcal{R}}\in\Es^{(n-1)\times(n-1)}$ are nonsingular upper block triangular matrices. Combining with \cref{eq:K_VDU}, we now have two factorizations of $\bcal{K}$, i.e.,
\begin{equation}\label{eq:two_fact}
\bcal{W}\bcal{Q}\underbrace{\begin{bmatrix}
\widehat{\bm{G}} & \bcal{I}_{n-1}\\
G_n & 0
\end{bmatrix}}_{\bcal{G}}\begin{bmatrix}I&\\&&\widehat{\bcal{R}}\end{bmatrix} = \bcal{V}\bcal{D}\bcal{U}, \quad \mbox{with}\ \  \widehat{\bm{G}}
:=
\begin{bmatrix}
G_1\\G_2\\\vdots\\G_{n-1}
\end{bmatrix}.
\end{equation}

The right-hand side of \cref{eq:two_fact} is in the form of the $QR$-decomposition; we study that of the left-hand side by looking at the structure of the Cholesky factorization of $\bcal{G}^*\bcal{G}$. We observe that
\begin{align}\label{eq:cholesky}
\begin{bmatrix}
\widehat{\bm{G}}^* & G_n^*\\
\bcal{I} & 0
\end{bmatrix}\begin{bmatrix}
\widehat{\bm{G}} &\bcal{I}\\
G_n & 0
\end{bmatrix}
=
\begin{bmatrix}
\ssxy{F_0}{F_0} & \widehat{\bm{G}}^*\\
\widehat{\bm{G}}&\bcal{I}
\end{bmatrix}
=
\begin{bmatrix}
F_0^* & \\
\widehat{\bm{G}}F_0^{-1}&\bcal{R}_{G}^*
\end{bmatrix}
\begin{bmatrix}
F_0 & F_0^{-*}\widehat{\bm{G}}^*\\
&\bcal{R}_{G}
\end{bmatrix},
\end{align} 
where $\bcal{R}_{G}$ is the unique upper triangular Cholesky factor of $(\bcal{I}-\widehat{\bm{G}}F_0^{-1}F_0^{-*}\widehat{\bm{G}}^*)$, which is positive definite since $\no{\widehat{\bm{G}}F_0^{-1}}\preceq I$. There must be equality between the $R$-factors of the unique $QR$-decomposition of both sides of \cref{eq:two_fact}. Note that the matrix $\widehat{\bcal{R}}$ has arbitrary nonsingular block entries on the diagonal. To obtain the unique $R$-factor with entries in $\Es^+$, further transformation by a block diagonal unitary matrix (here denoted by $\widehat{\Gamma}$) is needed. Using convenient decompositions of the involved matrices, we obtain from \cref{eq:two_fact} that
\begin{align}
\begin{bmatrix}
F_0 & F_0^{-*}\widehat{\bm{G}}^*\\
&\bcal{R}_{G}
\end{bmatrix}
\begin{bmatrix}I&\\&\widehat{\bcal{R}}\end{bmatrix}
&= 
\begin{bmatrix}I&\\&\widehat{\Gamma}\end{bmatrix}
\bcal{D}\bcal{U}
\\
\begin{bmatrix}
F_0 & F_0^{-*}\widehat{\bm{G}}^*\\
&\bcal{R}_{G}
\end{bmatrix}
\begin{bmatrix}I&\\&\widehat{\bcal{R}}\end{bmatrix}
&= 
\begin{bmatrix}I&\\&\widehat{\Gamma}\end{bmatrix}
\begin{bmatrix}D_1&\\&\widehat{\bcal{D}}\end{bmatrix}
\begin{bmatrix}I&\bm{U}_{12}\\&\bcal{U}_{22}\end{bmatrix}
\\
\begin{bmatrix}
F_0 & F_0^{-*}\widehat{\bm{G}}^*\widehat{\bcal{R}}\\
&\bcal{R}_{G}\widehat{\bcal{R}}
\end{bmatrix}
&= 
\begin{bmatrix}D_1&D_1\bm{U}_{12}\\&\widehat{\Gamma}\widehat{\bcal{D}}\bcal{U}_{22}\end{bmatrix};
\end{align}
see also \cite[sec. 3]{DuintjerTebbens2012Any} for the analog in the standard case.

We proceed by comparing individual block entries. Equality of the first diagonal block entries gives 
\begin{equation}\label{eq:F0}
D_1 = F_0,
\end{equation}
which is satisfied trivially.
Equality of the second diagonal block entries gives an expression for $\widehat{\bcal{R}}$
\begin{equation}\label{eq:R_hat}
\widehat{\bcal{R}}=\bcal{R}_{G}^{-1}\widehat{\Gamma}\widehat{\bcal{D}}\bcal{U}_{22}.
\end{equation}
Substituting from \cref{eq:F0} and \cref{eq:R_hat} to the equation given by the off-diagonal entry, we obtain
\begin{align}\label{eq:D_cond}
F_0^{-*}\widehat{\bm{G}}^*\bcal{R}_{G}^{-1}\widehat{\Gamma}\widehat{\bcal{D}}\bcal{U}_{22} &= 
F_0\bm{U}_{12}
\\
\left(\ssxy{F_0}{F_0}^{-1}\widehat{\bm{G}}^*\bcal{R}_{G}^{-1}\right)\left(\widehat{\Gamma}\widehat{\bcal{D}}\right)
&= 
\bm{U}_{12}\bcal{U}_{22}^{-1}.
\end{align}

We now investigate the objects in equation \cref{eq:D_cond}.  First, using the block inversion formula on \cref{eq:U_Ritz}, we have  
\begin{equation}
\bm{U}_{12}\bcal{U}_{22}^{-1}= 
\begin{bmatrix}C^{(1)}_0 & \cdots & C^{(n-1)}_0\end{bmatrix}.
\end{equation}
Second, we apply \Cref{th:aux_lemma} to $\widehat{\bm{G}}F_0^{-1}$, and together with the definition of $\widehat{\bm{G}}$, we obtain 
\begin{align}
\bm{E}_k^T\bcal{R}_{G}^{-*}\widehat{\bm{G}}\,\ssxy{F_0}{F_0}^{-*}
& = \bm{E}_k^T\bcal{R}_{G}^{-*}\left(\widehat{\bm{G}}\,F_0^{-1}\right)F_0^{-*}\\
&= Q_k\sqrt{\ssxy{F_{k}}{F_{k}}^{-1} -\ssxy{F_{k-1}}{F_{k-1}}^{-1}}, \quad |Q_k| = I. 
\end{align}
This substituting to \cref{eq:D_cond} and defining $\widehat{\bcal{Q}} := \text{diag}(Q_1^*,\ldots,Q_{n-1}^*)$ gives
\begin{equation}\label{eq:consistency_mat_form}
\sqrt{\ssxy{F_{k}}{F_{k}}^{-1} -\ssxy{F_{k-1}}{F_{k-1}}^{-1}}\,^*\left(\bm{E}_k^T\widehat{\bcal{Q}}\widehat{\Gamma}\widehat{\bcal{D}}\bm{E}_k\right)
= C^{(k)}_0, \quad k=1,\ldots,n-1.
\end{equation}

We now have $n-1$ equations for the $n-1$ block entries of the block diagonal matrix $\widehat{\bcal{Q}}\widehat{\Gamma}\widehat{\bcal{D}}$.
To ensure that a non-singular $\widehat{\bcal{Q}}\widehat{\Gamma}\widehat{\bcal{D}}$ satisfying \cref{eq:consistency_mat_form} exists, there must be some consistency between $\ssxy{F_{k}}{F_{k}}^{-1} -\ssxy{F_{k-1}}{F_{k-1}}^{-1}$ and $C^{(k)}_0$. More precisely, it has to hold that, for $k = 1,\ldots,n-1$,
\begin{equation}\label{eq:range_cond}
\text{Range}\left(\sqrt{\ssxy{F_{k}}{F_{k}}^{-1} -\ssxy{F_{k-1}}{F_{k-1}}^{-1}}\,^*\right) = \text{Range}\left(C^{(k)}_0\right),
\end{equation}
or alternatively\footnote{Using the fact that $\text{Range}(R^\ast) = \text{Range}(R^\ast R)$.}
\begin{equation}\label{eq:range_cond_final}
\text{Range}\left(\ssxy{F_{k}}{F_{k}}^{-1} -\ssxy{F_{k-1}}{F_{k-1}}^{-1}\right) = \text{Range}\left(C^{(k)}_0\right).
\end{equation}

If $\text{Range}\left(C^{(k)}_0\right) = \text{Range}\left(\ssxy{F_{k}}{F_{k}}^{-1} -\ssxy{F_{k-1}}{F_{k-1}}^{-1}\right) = \Ce^s$, then the entry $\bm{E}_k^T\widehat{\bcal{D}}\bm{E}_k$ is defined uniquely as
\begin{align}
	\bm{E}_k^T\widehat{\bcal{D}}\bm{E}_k &= \left|\sqrt{\ssxy{F_{k}}{F_{k}}^{-1} -\ssxy{F_{k-1}}{F_{k-1}}^{-1}}^{-*}C^{(k)}_0\right|\\
	&= \sqrt{\left(C^{(k)}_0\right)^*\left(\ssxy{F_{k}}{F_{k}}^{-1} -\ssxy{F_{k-1}}{F_{k-1}}^{-1}\right)^{-1}C^{(k)}_0}.
\end{align}
In other cases satisfying \eqref{eq:range_cond_final}, there is certain freedom in the components corresponding to the null space of $\sqrt{\ssxy{F_{k}}{F_{k}}^{-1} -\ssxy{F_{k-1}}{F_{k-1}}^{-1}}\,^*$.

In the sense of \Cref{sec:peak_plateau}, \cref{eq:range_cond_final} implies that the FOM residual norm must satisfy
\begin{equation}
\text{Range}\left(\ssxy{\bm{R}_{k}^F}{\bm{R}_{k}^F}^{\dagger}\right) = \text{Range}\left(\ssxy{\bm{R}_{k}^F}{\bm{R}_{k}^F}\right) = \text{Range}\left(C^{(k)}_0\right).
\end{equation}

\subsection{Final result}
In the standard case, \cref{eq:range_cond_final} reduces to the well-known condition that GMRES stagnates if and only if we obtain at least one zero Ritz value. In the block case, (partial) stagnation also appears if and only if $\bcal{H}^{(k)}$ is singular. But in addition, the rank deficiency of $\ssxy{F_{k}}{F_{k}}^{-1} -\ssxy{F_{k-1}}{F_{k-1}}^{-1}$, encoding the stagnation, and the rank deficiency of $C^{(k)}_0$, encoding the singularity of $\bcal{H}^{(k)}$, must have the same structure, i.e., the corresponding matrices must share the same range. If this is the case, we can prescribe the convergence of the block Arnoldi method and the block GMRES method at the same time, as summarized in the following theorem, which generalizes \cite[Th. 3.6]{DuintjerTebbens2012Any}.

\begin{theorem}\label{th:gmres_arnoldi}
	Let $\left\{M^{(k)}\right\}_{k=1}^{n}$ be any sequence of $\lambda$-matrices,
	\begin{equation}
		M^{(k)}(\lambda) = \lambda^kI - \sum_{j=0}^{k-1}\lambda^jC_j^{(k)}, 
		\end{equation} 
	$C_0^{(n)}$ nonsingular, and let $\{F_k\}_{k=0}^{n-1}$, $F_k\in\Es^+$, be any sequence satisfying
	\begin{equation}
	F_0\succeq F_1 \succeq F_2 \succeq \cdots \succeq F_{n-1}\succ 0.
	\end{equation}
	Under the assumption that the two sequences satisfy the consistency condition
		\begin{equation}
		\text{Range}\left(\ssxy{F_{k}}{F_{k}}^{-1} -\ssxy{F_{k-1}}{F_{k-1}}^{-1}\right) = \text{Range}\left(C^{(k)}_0\right),
		\end{equation}
	the following two assertions are equivalent:
	\begin{enumerate}
		\item The residuals of blGMRES($\bcal{A}$,$\bm{B}$) satisfy 
		\begin{equation}
		\no{\bm{R}_k} = F_k, \quad k = 0,\ldots,n-1,
		\end{equation}
		and the principal submatrices of the Hessenberg matrix generated by blArnoldi($\bcal{A}$,$\bm{B}$) satisfy
		\begin{equation}
		M^{(k)}\left(\bcal{H}^{(k)}\right)\circ \bm{E}_1\no{\bm{B}} = 0, \quad k = 1,\ldots,n.
		\end{equation}
		\item The matrix $\bcal{A}$ and the starting vector/right-hand side $\bm{B}$ are of the form 
		\begin{equation}
		\bcal{A} =\bcal{V}\bcal{D}\bcal{U}\bcal{C}\bcal{U}^{-1}\bcal{D}^{-1}\bcal{V}^*, \quad \bm{B} = \bcal{V}\bm{E}_1F_0,
		\end{equation}
		where $\bcal{V}$ is a unitary matrix, $\bcal{C}$ is the block companion matrix corresponding to $M^{(n)}$,
		\begin{equation}
		\bcal{U} = \begin{bmatrix}
		I&-C_0^{(1)}& -C_0^{(2)}&\cdots&-C_0^{(n-1)}\\
		&I&-C_1^{(2)}& \cdots & \vdots\\
		& & \ddots & \ddots & \vdots\\
		& & & I& -C_{n-2}^{(n-1)}\\
		& & & & I\\
		\end{bmatrix}^{-1},
		\end{equation}
		and $\bcal{D}$ is a block diagonal matrix with entries in $\Es^+$ satisfying
		\begin{align}\label{eq:D_cond_theorem}
		D_{1}&= F_0,\\
		\sqrt{\ssxy{F_{k-1}}{F_{k-1}}^{-1} -\ssxy{F_{k-2}}{F_{k-2}}^{-1}}\,^*\left(Q_{k}D_{k}\right)&= C^{(k-1)}_0, \quad k=2,\ldots,n,
		\end{align}
		for some $Q_k\in\Es$, $|Q_k| = I$.
	\end{enumerate}
\end{theorem}
\begin{proof}
	The proof follows from the construction in \Cref{sec:spectral_A,sec:spectral_H,sec:Arnoldi_vs_GMRES}.
\end{proof}
\Cref{th:gmres_arnoldi} shows that, similar to the standard case, we can prescribe the convergence behavior for blArnoldi independently of the residual convergence of blGMRES, as long as the (partially) stagnating iterations are reflected in the corresponding $\lambda$-matrix and vice versa.

\begin{remark}
If we do not prescribe the blGMRES convergence, we only have a condition on $\bcal{U}$, and  $\bcal{D}$ can be an arbitrary block diagonal matrix with entries in $\Es^+$. Similarly, if the blArnoldi convergence is not prescribed, \cref{eq:D_cond_theorem} only gives a condition on the first row of $(\bcal{D}\bcal{U})^{-1}$, which has to satisfy \begin{align}
\bm{E}_1^T(\bcal{D}\bcal{U})^{-1}\bm{E}_1 &=F_0^{-1} \\
\bm{E}_1^T(\bcal{D}\bcal{U})^{-1}\bm{E}_k &= \sqrt{\ssxy{F_{k-1}}{F_{k-1}}^{-1} -\ssxy{F_{k-2}}{F_{k-2}}^{-1}}\,^*Q_{k}, \quad k=2,\dots,n,
\end{align}
cf. \cite[Th. 1]{DuintjerTebbens2014Prescribing}.
\end{remark}
\begin{remark} The Ritz value companion transform $\bcal{U}$ makes the intermediate $\lambda$-matrices $M^{(k)}$ completely independent of the last $\lambda$-matrix $M^{(n)}$ and these are all independent of the matrix $\bcal{D}$ defining the subdiagonal entries of $\bcal{H}$. In the standard case, this implies that the (often used) residual measure of the Ritz value convergence	
	\begin{equation}\label{eq:ritz_residual}
	\left\|AV^{(k)}\bm{z}^{(k)}_i - V^{(k)}\bm{z}^{(k)}_i\theta^{(k)}_i\right\| = \left|h_{k+1,k}\bm{e}_k^T\bm{z}^{(k)}_i\right|, \quad \left(\theta^{(k)}_i,\bm{z}^{(k)}_i\right) \ \text{an eigenpair of }H^{(k)},
	\end{equation}
	may provide little or no useful information about the convergence of the Ritz values to the eigenvalues of $A$; see \cite[pp. 964--965]{DuintjerTebbens2012Any}. With the block generalization of the Jordan form, see, e.g., \cite{Gohberg2009Matrix}, a similar conclusion is possible for the block version. Further analysis of this topic is however beyond the scope of this paper.
\end{remark} 

In the next section, we discuss some of the fundamental differences between the results presented in the preceding papers on standard Arnoldi and GMRES, and the results regarding their block counterparts presented here.

\subsection{The role of polynomials with matrix coefficients}\label{sec:companion}
In \Cref{sec:prescribing_conv}, the spectral properties of $\bcal{A}$ and the submatrices $\bcal{H}^{(k)}$ are prescribed through the $\lambda$-matrices $M^{(k)}(\lambda)$. The relation \cref{eq:char_poly}, equivalent to \cref{eq:AK_KC}, together with the assumption that $\mathcal{K}_n(\bcal{A},\bm{B})$ is of full rank, means that $M(\bcal{A})$ is zero when evaluated on $n$ linearly independent block vectors, as it satisfies
\begin{equation}\label{eq:MA_circ_V_i}
M(\bcal{A})\circ \bm{V}_i = 0 \quad \text{for} \ \ \bm{V}_i = \bcal{A}^{i-1}\bm{B}, \quad i = 1,\ldots,n.
\end{equation}

In the standard case, \cref{eq:MA_circ_V_i} implies $M(A)=0$, i.e., $M$ is the characteristic polynomial of $A$, and $M(A)\circ \bm{v} = 0$, $\forall \bm{v}\in\Ce^n$. This can be also seen from
\begin{equation}
M(A)\circ \left(\sum_{i = 1}^n\bm{v}_id_i\right) = \sum_{i = 1}^n\overbrace{(M(A)\circ \bm{v}_i)}^{=0}d_i =0.
\end{equation} 
In other words, requiring $M(A)\circ \bm{b} = 0$ is equivalent to prescribing the eigenvalues of $A$.

In the block case, the situation is different. First, $M(\bcal{A})$ itself cannot be defined, because of the clash of dimensions. Further, despite the fact that any vector $\bm{V}\in\Es^n$ can be written as a block linear combination of $\bm{V}_1,\ldots,\bm{V}_n$, i.e., $\bm{V} = \sum_{i = 1}^n\bm{V}_iD_i$, \cref{eq:MA_circ_V_i} does \underline{not} imply $M(\bcal{A})\circ \bm{V} = 0$, $\forall \bm{V}\in\Es^n$. This is because 
\begin{align}
M(\bcal{A})\circ \left(\sum_{i = 1}^n\bm{V}_iD_i\right)  &= \bcal{A}^n\left(\sum_{i = 1}^n\bm{V}_iD_i\right) - \sum_{k=0}^{n-1}\bcal{A}^k\left(\sum_{i = 1}^n\bm{V}_iD_i\right)C_k\\ &= \sum_{i = 1}^n \left(\bcal{A}^n\bm{V}_iD_i - \sum_{k=0}^{n-1}\bcal{A}^k\bm{V}_iD_iC_k\right)\label{eq:DC}\\
&\neq \sum_{i = 1}^n \left(\bcal{A}^n\bm{V}_iD_i - \sum_{k=0}^{n-1}\bcal{A}^k\bm{V}_iC_kD_i\right) = \sum_{i = 1}^n\overbrace{(M(\bcal{A})\circ \bm{V}_i)}^{=0}D_i =0\label{eq:CD}.
\end{align}
The transition between \cref{eq:DC} and \cref{eq:CD} is only possible when $D_i$ and $C_k$ commute for all $i$ and $k$. This makes the standard case different from the block case.

From \cref{eq:A_KCK}, it is clear that the eigenvalues of $\bcal{A}$ are defined by the eigenvalues of $\bcal{C}$ and coincide with the latent roots of $M$; see also \Cref{sec:poly_mat}. For a given block companion matrix ${\bcal{C}\in\Es^{n\times n}}$, the manifold of block companion matrices similar to $\bcal{C}$ has dimension $ns^2-ns$; see, e.g., \cite{Edelman1995Polynomial}. Therefore the eigenvalues do not define $M$ uniquely. This can be seen also from the fact that the coefficients $C_0,\ldots,C_{n-1}$ have $ns^2$ free parameters, while there only are $ns$ eigenvalues. 

Instead of focusing on the eigenvalues of $\bcal{A}$, it is advantageous to remain in the block setting and look at the solvents of the $\lambda$-matrix instead; see \Cref{sec:poly_mat}. Prescribing the (chain of) solvents $S_1,\ldots,S_n$ of the $\lambda$-matrix \cref{eq:lambda_matrix_II}, the coefficients of the $\lambda$-matrix $M$ are defined uniquely through \cref{eq:roots_to_coeff}. The eigenvalues of $\bcal{A}$ are the eigenvalues of the solvents, but changing the eigenvectors of the individual solvents will have impact on the eigenvectors of $\bcal{A}$. To achieve $C_0$ nonsingular, we only require that each of the solvents $S_1,\ldots,S_n$ of $M$ is nonsingular. A similar result holds for the submatrices $\bcal{H}^{(k)}$, with the exception that singular solvents are allowed.

Concluding, in the standard case, if the eigenvalues of $\bcal{A}$ and the convergence behavior of both GMRES and Arnoldi are prescribed, then, in case of no stagnation of GMRES, all matrices and right-hand sides satisfying these conditions will be identical up to a unitary transform. Each step in which GMRES stagnates provides one extra free parameter, represented by the entry of the diagonal matrix $\bcal{D}$ in \cref{eq:D_cond_theorem}. In the block case, prescribing the eigenvalues of $\bcal{A}$ and the convergence behavior of both blGMRES and blArnoldi will still give us certain freedom in the choice of the eigenvectors of $\bcal{A}$, plus again some extra free parameters in the stagnating iterations. From this point of view, the eigenvalues of $\bcal{A}$ are even less indicative regarding the residual convergence behavior of blArnoldi and blGMRES than they are in the standard case.

\section{Conclusions and open questions}\label{sec:conclusion}
The analysis of block Krylov subspace methods has always presented a challenge beyond those encountered with classical non-block methods.  This is due to the interaction between the right-hand sides.  We have demonstrated here that the $^{\ast}$-algebra approach introduced  in \cite{Frommer2017Block} enables us to cleanly obtain the same sort of results one sees for non-block Krylov subspace methods, which have previously been unavailable.  These results fill a certain gap in the understanding of the convergence behavior of block Krylov methods for non-symmetric matrices. 

We have thus obtained block versions of the fundamental results regarding the admissible and attainable convergence of standard Arnoldi and GMRES presented in a series of papers published over the last 25 years. Extending the framework introduced in \cite{Frommer2017Block}, we were able to keep the formal notation as close as possible to the original results, with the block generalization of the Givens transformation allowing for extension of well-known textbook relations for the residuals of GMRES and FOM to the block case. 
This framework allows us to see the interdependence of the residual convergence behaviors of the individual systems.
We explicitly formulated conditions on the admissible convergence behavior of the residuals of blGMRES using an appropriate block generalization of the norm. 

Under the assumption that blGMRES does not converge prematurely, we were then able to completely characterize matrices and right-hand sides producing any prescribed admissible convergence behavior. Furthermore, spectral properties of the matrix can be enforced through the similarity to a block companion matrix. We showed that arbitrary convergence of blArnoldi for the eigenvalue problem is possible, and that arbitrary convergence of blGMRES and blArnoldi can be, under moderate assumptions, achieved simultaneously. Combining these results with the theory of block companion matrices, we showed that in a certain sense, increasing the number of right-hand sides reduces the predictive value of the eigenvalues of $\bcal{A}$.

It should be noted that, as Meurant pointed out in \cite{Meurant2012GMRES}, these results all concern \textit{residual} convergence behavior.  It is observed that for a class of matrices (each with different spectral properties) constructed to exhibit a specific admissible residual convergence behavior, the actual error convergence behavior may very well exhibit dependence on spectral properties of the matrix.  Thus, one interpretation of the results presented in this paper and the work of the last 25 years on this topic is that spectral properties of a non-Hermitian matrix may not a priori tell us much about the behavior of GMRES with respect to \textit{our chosen method} of measuring convergence.  Thus one may connect the results in this and related works to the notion that one should measure (residual) error using an appropriate norm.  However, this is beyond the scope of the current paper. 

Certain important aspects of blGMRES and blArnoldi convergence also are beyond the scope of this paper. Since the Arnoldi algorithm is based on full orthogonalization, the amount of data that needs to be stored, as well as the computational complexity, grows with each iteration. In practical computations, restarting the orthogonalization process is therefore often unavoidable. For this reason, analysis of the influence of the restarts on the admissible convergence represents another important research direction. Furthermore, the blArnoldi process breaks down when fewer than $s$ linearly independent Arnoldi vectors are generated in the $k$th step. The situation when no single system has converged but rather a linear combination of the columns of $\bm{X}$ lies in $\mathcal{K}_k(\bcal{A},\bm{B})$ is particularly unpleasant. In the considered framework, resolving this situation by reducing the block size is not directly possible, since by this, we change $\Es$. To avoid change in the block size, the linearly dependent vectors can be replaced by some auxiliary (random) vectors. This direction will be further explored elsewhere.

\appendix

\section{Polynomials with matrix coefficients}\label{sec:poly_mat}
In this section, we recall some terminology and fundamental results from the theory of polynomials with matrix coefficients. Statements presented in this section are adopted from \cite{Dennis1976algebraic}.  For a more comprehensive overview of the topic, we recommend \cite{Gohberg2009Matrix}. 

Let $C_k\in\Ce^{s\times s}$, $k= 0,\ldots,n-1$. We call
\begin{equation}
M(\lambda) = \lambda^nI - \sum_{k=0}^{n-1}\lambda^kC_k, \quad \lambda\in\Ce,\label{eq:lambda_matrix}
\end{equation} 
a \emph{$\lambda$-matrix} and
\begin{equation}
M(X) = X^n - \sum_{k=0}^{n-1}C_kX^k, \quad X\in\Ce^{s\times s}, \label{eq:matrix_polynomial}
\end{equation} 
a \emph{matrix polynomial}. 
We call $\lambda\in\Ce$ a \emph{latent root} of the $\lambda$-matrix $M$ if $M(\lambda)$ is singular, and $S\in\Ce^{s\times s}$ a \emph{right solvent} of the matrix polynomial $M$ if $M(S) = 0$. The latent roots and solvents are related as follows. 
\begin{theorem}
	If $S$ is a solvent of the matrix polynomial $M$, then the $\lambda$-matrix $M$ can be factorized as
	\begin{equation}
	M(\lambda) = Q(\lambda)(I\lambda - S).
	\end{equation}
\end{theorem}
\begin{corollary}
	The $s$  eigenvalues of the solvent $S$ of the matrix polynomial $M$ are all latent roots of the $\lambda$-matrix $M$.
\end{corollary}

A sequence of matrices $S_1,\ldots,S_n$ forms a \emph{chain of solvents} if
\begin{equation}
M(\lambda)= (I\lambda - S_1)(I\lambda - S_2)\cdots(I\lambda - S_n).\label{eq:chain_of_solvents}
\end{equation} 
It follows directly from \cref{eq:chain_of_solvents}  that
\begin{align}
C_{n-1} &= S_1+S_2+\cdots + S_n,\\
C_{n-2} &= - (S_1S_2 + S_1S_3+\cdots+ S_{n-1}S_n), \\\label{eq:roots_to_coeff}
&\vdots\\
C_0 &= (-1)^{n-1} S_1S_2\cdots S_n.
\end{align}
The matrix $C$, which has the form
\begin{equation}
C = 
\begin{bmatrix}
0& && C_0\\
I&\ddots&&C_1\\
&\ddots&0&\vdots\\
& & I &C_{n-1}
\end{bmatrix}
\end{equation}
is called the \emph{block companion matrix} associated with the $\lambda$-matrix \cref{eq:lambda_matrix} or equivalently the matrix polynomial \cref{eq:matrix_polynomial}. Eigenvalues of the block companion matrix $C$ and the latent roots are related as follows.
\begin{theorem}
	$\det(C-\lambda I) = (-1)^{ns}\det(I\lambda^n - C_{n-1}\lambda^{n-1}-\cdots - C_0)$. 
\end{theorem}
\begin{corollary}
	The eigenvalues of the block companion matrix are the latent roots of the associated $\lambda$-matrix, therefore $M$ has exactly $ns$ latent roots.\end{corollary} 

We define the action of a matrix polynomial on a block vector as
\begin{equation}\label{eq:circ_operation}
M(A)\circ V = A^n\,V - \sum_{k=0}^{n-1}A^k\,V\,C_k, \quad A\in\Ce^{m\times m}, \ \ V\in\Ce^{m\times s},
\end{equation}
see \cite[p. 108]{Simoncini1996Convergence} or \cite[p. 107]{Frommer2017Block}.

\section{Auxiliary lemma}
We use MATLAB notation in this lemma. In particular, $Z_j$ denotes the $j$th block entry and $\bm{Z}_{1:j}$ denotes the first $j$ block entries of the block vector $\bm{Z}$.

\begin{lemma}\label{th:aux_lemma}
	Let $\bm{Z}\in\Es^k$ be such that $\bcal{I}-\bm{Z}\bm{Z}^*$ is positive definite, and let $\bcal{R}_Z$ be the unique upper triangular Cholesky factor of $\bcal{I}-\bm{Z}\bm{Z}^*$. Then 
	\begin{equation}
	|\bm{E}_j^T\bcal{R}_{Z}^{-*}\bm{Z}| = \sqrt{(I-\bm{Z}_{1:j}^*\bm{Z}_{1:j})^{-1} - (I-\bm{Z}_{1:j-1}^*\bm{Z}_{1:j-1})^{-1}}, \quad j = 1,\ldots,k.
	\end{equation}
\end{lemma}
\begin{proof}
	Since $\bcal{R}_{Z}^{-*}$ is lower triangular, it holds that	
	\begin{equation}
	\bm{E}_j^T\bcal{R}_{Z}^{-*}\bm{Z} = \bm{E}_j^T\bcal{R}_{Z_{1:j}}^{-*}\bm{Z}_{1:j}.
	\end{equation}
	Therefore, it suffices to investigate $\bm{E}_k^T\bcal{R}_{Z}^{-*}\bm{Z}$ and apply the lemma recursively.
	
	We observe that
	\begin{align}
	\bm{E}_k^T\bcal{R}_{Z}^{-*}\bm{Z} = \bm{E}_k^T\bcal{R}_{Z}(\bcal{R}_{Z}^*\bcal{R}_{Z})^{-1}\bm{Z} &= \bm{E}_k^T\bcal{R}_{Z}(\bcal{I}-\bm{Z}\bm{Z}^*)^{-1}\bm{Z}\\ 
	&= \bm{E}_k^T\bcal{R}_{Z}\bm{Z}(I-\bm{Z}^*\bm{Z})^{-1}\\
	& = \bm{E}_k^T\bcal{R}_{Z}\bm{E}_k\bm{E}_k^T\bm{Z}(I-\bm{Z}^*\bm{Z})^{-1},
	\end{align}
	where we used the push-through identity (Woodbury matrix identity) and the fact that $\bcal{R}_Z$ is upper triangular.
	
	The block entry $\bm{E}_k^T\bcal{R}_{Z}\bm{E}_k$, i.e., the $(k,k)$ block entry of the Cholesky factor, can be obtained as
	\begin{equation}
	\bm{E}_k^T\bcal{R}_{Z}\bm{E}_k = \sqrt{I-Z_{k}(I-\bm{Z}_{1:k-1}^*\bm{Z}_{1:k-1})^{-1}Z_{k}^*}.
	\end{equation}
	Thus
	\begin{align}
	&\ssxy{\bm{E}_k^T\bcal{R}_{Z}^{-*}\bm{Z}}{\bm{E}_k^T\bcal{R}_{Z}^{-*}\bm{Z}}\\ 
	&\hspace{1.8cm}= (I-\bm{Z}^*\bm{Z})^{-*}Z_{k}^*\left(I-Z_{k}(I-\bm{Z}_{1:k-1}^*\bm{Z}_{1:k-1})^{-1}Z_{k}^*\right)Z_{k}(I-\bm{Z}^*\bm{Z})^{-1}\\
	&\hspace{1.8cm}= (I-\bm{Z}^*\bm{Z})^{-1}Z_{k}^*Z_{k}\left(I-(I-\bm{Z}_{1:k-1}^*\bm{Z}_{1:k-1})^{-1}Z_{k}^*Z_{k}\right)(I-\bm{Z}^*\bm{Z})^{-1}\\
	&\hspace{1.8cm}= (I-\bm{Z}^*\bm{Z})^{-1}Z_{k}^*Z_{k}(I-\bm{Z}_{1:k-1}^*\bm{Z}_{1:k-1})^{-1}\left(I-\bm{Z}^*\bm{Z}\right)(I-\bm{Z}^*\bm{Z})^{-1}\\
	&\hspace{1.8cm}= (I-\bm{Z}^*\bm{Z})^{-1}\left((I-\bm{Z}_{1:k-1}^*\bm{Z}_{1:k-1})- (I-\bm{Z}^*\bm{Z})\right)(I-\bm{Z}_{1:k-1}^*\bm{Z}_{1:k-1})^{-1}\\
	&\hspace{1.8cm}=(I-\bm{Z}^*\bm{Z})^{-1} - (I-\bm{Z}_{1:k-1}^*\bm{Z}_{1:k-1})^{-1},
	\end{align}
	which gives the desired statement.
\end{proof}
Note that $\bcal{I}-\bm{Z}\bm{Z}^*$ is positive definite if and only if $\bcal{I}-\bm{Z}^*\bm{Z}$ is positive definite.

\section*{Acknowledgments}
The authors would like to thank G\'{e}rard Meurant for his valuable suggestions regarding presentation of the result of \Cref{sec:spectral_A,sec:spectral_H,sec:Arnoldi_vs_GMRES}. We are also grateful to two anonymous referees and the handling editor
for their comments and suggestions.

\end{document}